\documentclass{amsart}
\usepackage{amsmath, amsthm, amsopn, amsfonts, verbatim}
\usepackage{color}
\usepackage{epsfig}
\usepackage{amssymb}
\usepackage{graphicx}

\usepackage[T1]{fontenc}

\usepackage{epic, eepic}

\newcommand{\br}{\mathbb R}

\def\squarebox#1{\hbox to #1{\hfill\vbox to #1{\vfill}}} 
 
\newcommand{\1}{{\bold 1}}

\newcommand{\RR}{{\mathbb R}} 
\newcommand{\SP}{{\mathbb S}}

\newcommand{\cyl}{{\mathcal C}}

\newcommand{\supp}{\operatorname{supp}}

\newcommand{\eps}{{\varepsilon}} 
\newcommand\ESCS{ESCS}

\theoremstyle{plain}

\newtheorem{thm}{Theorem}

\newtheorem{prop}[thm]{Proposition}

\newtheorem{lem}[thm]{Lemma}

\theoremstyle{definition}

\theoremstyle{remark} \newtheorem{rem}[thm]{Remark}

\newtheorem{defn}[thm]{Definition} 

\numberwithin{equation}{section}

\ifx\pdfoutput\undefined
  \DeclareGraphicsExtensions{.eps}
\else
  \ifx\pdfoutput\relax
    \DeclareGraphicsExtensions{.eps}
  \else
    \ifnum\pdfoutput>0
      \DeclareGraphicsExtensions{.pdf}
    \else
      \DeclareGraphicsExtensions{.eps}
    \fi
  \fi
\fi

\title[Eigenfunction concentration for polygonal billiards]
{Eigenfunction concentration for polygonal billiards}
\author[A. Hassell]{Andrew Hassell}
\author[L. Hillairet]{Luc Hillairet}
\author[J. Marzuola]{Jeremy Marzuola}

\address{Mathematics Department, Australian National University \\ Canberra, ACT  0200, Australia}
\email{hassell@maths.anu.edu.au}
\address{UMR CNRS 6629-Universit\'{e} de Nantes, 2 rue de la Houssini\`{e}re, \\
BP 92 208, F-44 322 Nantes Cedex 3, France}
\email{Luc.Hillairet@math.univ-nantes.fr}
\address{Applied Mathematics Department, Columbia University \\
200 S.W. Mudd, 500 W. 120th St., New York, NY 10027, USA}
\email{jm3058@columbia.edu}

\keywords{Polygonal billiards, eigenfunction concentration, semiclassical measures, control region}
\subjclass[2000]{35P20}
\thanks{A.H. was partially supported by Discovery Grant DP0771826
from the Australian Research Council.  J.M. was supported by a National Science Foundation Postdoctoral Fellowship and would like to thank Australian National University for generously hosting him at time of the beginning of this research.}

\usepackage{amssymb}
\usepackage{amsmath, amsthm, amsopn, amsfonts}

\def\11{{\rm 1~\hspace{-1.4ex}l} }
\def\R{\mathbb R}

\begin{document}    
   
\begin{abstract} In this note, we extend the results on eigenfunction concentration in billiards 
as proved by the third author in \cite{M1}.  
There, the methods developed in Burq-Zworski \cite{BZ3} 
to study eigenfunctions for billiards which have rectangular components were applied.  
Here we take an arbitrary polygonal billiard $B$ and show that eigenfunction mass cannot 
concentrate away from the vertices; in other words, given any neighbourhood $U$ of the vertices, 
there is a lower bound 
$$
\int_U |u|^2 \geq c \int_B |u|^2
$$
for some $c = c(U) > 0$ and any eigenfunction $u$. 
%The results apply also to plane domains, or tori, with slits, with pseudointegrable billiard flow. 
\end{abstract}

\maketitle   

\section{Introduction}
\label{one}

Let $B$ be a plane polygonal domain, not necessarily convex. Let $V$ denote the set of vertices of $B$, and let 
$\Delta_B$ denote the Dirichlet or the Neumann Laplacian on $L^2(B)$. In this note, we will prove the following 

\begin{thm}\label{ThmPoly} 
Let $B$ be as above and let $U$ be any neighbourhood of $V$. Then there exists $c=c(U)> 0$ 
such that, for any $L^2$-normalized eigenfunction $u$ of the 
Dirichlet (or Neumann) Laplacian $\Delta_B$, we have
\begin{equation*}
\int_U |u|^2 \geq c.
\end{equation*}
That is, $U$ is a control region for $B$, in the terminology of \cite{BZ2}. 
\end{thm}

We will generalize this result to any Euclidean surface with conical singularities 
(\ESCS) $X$ with $U$ being any neighbourhood of the set of conical points and the 
$u_k$ being the eigenfunctions of the (Friedrichs)-Laplace operator on $X$
(see Section \ref{two} for more precise definitions).

Actually our main concern will be to derive a  sufficient geometric condition for $U\subset X$ 
that ensures that $U$ is a control region for the Laplace operator. 
It is well-known that such a condition is obtained by the so-called geometric control 
(see \cite{BZ2} for instance) and we will be interested in regions $U$ for which geometric control fails. 
One major obstruction to geometric control is the existence of periodic orbits that 
do not intersect $U.$ Since, on an \ESCS, the non-singular periodic orbits form 
Euclidean cylinders immersed in $X$, we introduce the 
following geometric condition:

\begin{defn}
A region $U\subset X$ is said to satisfy  condition $(CC)$ (the `cylinder condition')  if the following 
two properties hold.
\begin{enumerate}
\item Any orbit that avoids $U$ is non-singular and periodic.
\item  There exists a finite collection of cylinders $(\cyl_i)_{i\leq N}$ such that 
any orbit that avoids $U$ belongs to some $\cyl_i$.
\end{enumerate}
\end{defn}

Here, by a \emph{cylinder} we mean an isometric immersion of $ \SP^1_l \times I$ into $X_0$, 
where $I \subset \RR$ is an interval and $\SP^1_l$ is the circle of length $l$ (see Lemma 
\ref{defcyl} below). 

The main theorem of this paper is then the following.

\begin{thm}
\label{ThmMain}
Let $X$ be an orientable \ESCS \ and $U$ a domain satisfying $(CC)$. Then 
there exists a positive constant $c=c(U)$ such that any normalized eigenfunction $u_k$ of the 
Euclidean Laplace operator on $X$ satisfies 
\begin{equation}
\label{est}
\int_U |u_k|^2 \geq c
\end{equation}
\end{thm}

The first theorem is derived from this one by letting $X$ be the double of the polygon $B$ and by taking 
$U$ the $\eps$ neighbourhood of the conical points of $X$, corresponding to the vertices of $B$; see Figure \ref{fig:3}. 
The fact that, for billiards,  $U$ satisfies $(CC)$ is 
the principal result of \cite{DG} (see Section \ref{four} below).

\begin{figure}
\includegraphics[width=4.0in]{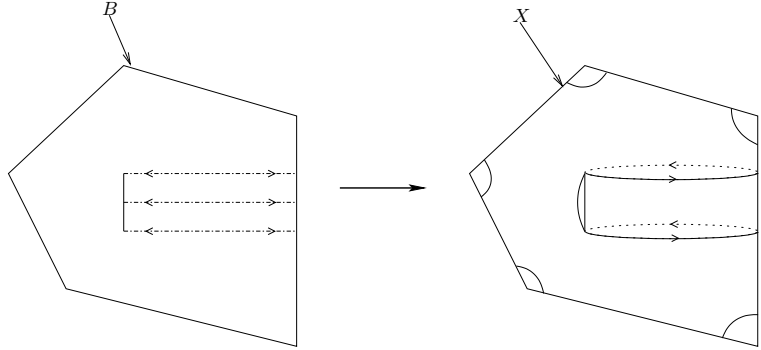}
\caption{Doubling a billiard $B$, here with a slit, to form an ESCS  $X$. Each vertex with angle $\alpha$ gives rise to a conic point with angle $2\alpha$, and the endpoints of the slit become conic points with angle $4\pi$.}
\label{fig:3}
\end{figure}

\begin{rem} Using Theorem \ref{ThmMain}, Theorem \ref{ThmPoly} can be sharpened 
by removing from $V$ those vertices with angle of the form $\pi/n$ for some integer $n,$  
as a reflection principle argument takes care of such vertices. 
Concretely, the geodesic flow may be non-ambiguously prolongated at such points. 
In particular, a cylinder hitting such a vertex on its boundary may be prolongated. 
\end{rem}

\begin{rem}
Let us mention other settings where our theorems apply. In the first theorem, 
the polygon may have polygonal holes and/or slits in it. In this case, $V$ should 
include the vertices of the holes and the ends of the slits. It also applies 
to translation surfaces, and thus also to tori with slits.  
%See Figure \ref{fig:1} for examples of applicable billiards.
\end{rem}

\begin{rem}
We will prove in Section \ref{three} that any neighbourhood of the conical 
points satisfies condition $(CC)$ on any 
\ESCS \  so that we could have stated the theorem without refering to this condition. 
We have stated it thus in order to emphasize the fact that such a control result usually 
requires two distinct steps. The first step
is to find some geometric or dynamical condition that implies 
control, and the second step is to find settings where this condition holds. 
These two steps proceed from essentially different methods: analytic in the 
case of the first step, and geometric/dynamical in the case of the second. 
\end{rem}

%\begin{figure}
%\includegraphics[width=6.0in]{fig1.ps}
%\caption{Examples of polygonal billiards for which Theorem \ref{ThmMain} is applicable with Dirichlet or Neumann boudary conditions on the solid lines and periodic boundary conditions on the dashed lines.}
%\label{fig:1}
%\end{figure}

The organization of the paper reflects the two steps of the proof described in the 
preceding remark. We first recall in Section~\ref{two} some basic facts about 
\ESCS s, semiclassical measures and the doubling procedure that 
allows one to treat polygonal billiards. 

In Section \ref{three}, we will prove 
that $(CC)$ implies control. To do this, we shall argue by contradiction 
in the following way. Assume that there is no $c$ such that \eqref{est} holds. 
Then there is a sequence $(u_n)$ of normalized eigenfunctions with eigenvalues $\lambda_n^2 \to \infty$, 
whose mass in $U$ tends to zero as $n \to \infty$. Associated to such a sequence is (at least one) semiclassical 
measure $\mu$ which is necessarily supported away from the inverse image $\pi^{-1}(U)$ --- 
see Lemma~\ref{supp}. It is a standard property of such semiclassical measures that the support of $\mu$ is invariant 
under the billiard flow. 
We shall show that the support property of $\mu$ just mentioned and the geometric condition $(CC)$ 
imply that $\mu$ would have to be supported on 
the cylinders $\cyl_i$. So it suffices, for a contradiction, to show that its mass on each such cylinder is zero. 
To do this we use the argument of \cite{M1} (which in turn relies on \cite{BZ3}) slightly modified so as 
to avoid a technical assumption made there.  

Finally, in Section~\ref{four}, we demonstrate $(CC)$ for any neighbourhood of the set $P$ of  conic points of an \ESCS \ $X$. 

%%%%%%%%%%%%%%%%%%%%%%%%%%%%%%%%%%

\section{\ESCS s, polygons and semiclassical measures}
\label{two}
%%%%%% Basic definitions of E.s.c.s %%%%%%%%%%%%
A Euclidean surface with conical singularities (ESCS) is a surface $X$ 
equipped with a metric $g$ such that $X$ may be written $X_0 \cup P$ where 
the metric $g$ is Euclidean on $X_0,$ and $P$ consists of a finite 
number of points $p_i$, such that each $p_i$ has a neighbourhood  isometric to a Euclidean cone whose tip 
corresponds to $p_i$. 

A reason for studying the Laplace operator on a \ESCS \ is its relation with polygonal billiards. 
Indeed, starting from a polygon $B$, possibly with polygonal holes and/or slits, the following doubling 
procedure gives a \ESCS \ $X.$ 
Take two copies $B$ and $\sigma B$ of the polygon where $\sigma$ is a reflection of the plane. The double 
$X$ is obtained by considering the formal union $P\cup \sigma P$ where two corresponding sides are 
identified pointwise. The reflection $\sigma$ then gives an involution of $X$ that commutes 
with the Laplace operator. The latter thus decomposes into odd and even 
functions and the reduced operators are then equivalent to 
the Laplace operator in $P$ with Dirichlet and Neumann boundary condition respectively. 
In particular, for any $u_n$ eigenfunction of the Neumann, resp. 
Dirichlet Laplace operator in $P$, we can construct an 
eigenfunction of the Laplace operator in $X$ by taking $u$ in $P$ and 
$u \circ \sigma$, resp. $-u \circ \sigma$ in $\sigma P.$

On such a surface, we shall consider the geodesic flow induced by the 
Euclidean metric on $X_0.$ We will not consider here the geodesics that end in 
a conical point since the condition $(CC)$ only considers  
non-singular geodesics. The following lemma  
shows that a non-singular periodic geodesic on a \ESCS \ is always part of a family.   

\begin{lem}\label{defcyl}
Let $X$ be an orientable \ESCS. 
Let $g\,: \R \rightarrow X$ be a non-singular T-periodic geodesic, then there exists 
$\delta>0$ such that $g$ extends to $h$ from $\R \times (-\delta,\delta) $ into $X_0$ 
such that 
\begin{enumerate}
\item $h(t,0)=g(t),$   
\item $h$ is a local isometry from $\R \times (-\delta,\delta)$ equipped with the flat metric into $X_0$,
\item $h$ is $T$-periodic in $t$. 
\end{enumerate}
Thus $h$ may be viewed as defined on the cylinder 
$\cyl_{\delta,T}\,:=\,\SP^1_T\times (-\delta,\delta)$. 

\end{lem}

\begin{proof}
Let $T$ be the smallest period of $g$.  For any $t  \leq T$ there exists $\delta_t$ such that the square 
$(-\delta_t,\delta_t)^2$ is isometric to a neighbourhood of $g(t).$ Moreover, this isometry 
$h_t$ may be chosen so that the horizontal segment $(-h_t,h_t) \times \{ 0\}$ 
projects onto $g(t-h_t,t+h_t)\,:\,h_t(t_1,0)=\gamma(t+t_1).$  Using compactness, 
$\delta\,=\,\inf \{ \delta_t,~t\in [0,T] \}$ exists and is positive. 
Gluing the $h_t$ by continuity defines a local 
isometry $h : \R\times (-\delta,\delta)$ into $X_0.$ 
By construction, for any $s$, $h(t,s)$ is at distance $|s|$ of the geodesic $g$ and 
this distance is realized by $g(t).$ Thus, there are 
only two possible choices for $h(t+T,s).$ Since $X$ is orientable, 
necessarily $h(t+T,s)=h(t,s).$ 
\end{proof}

\begin{figure}
\includegraphics{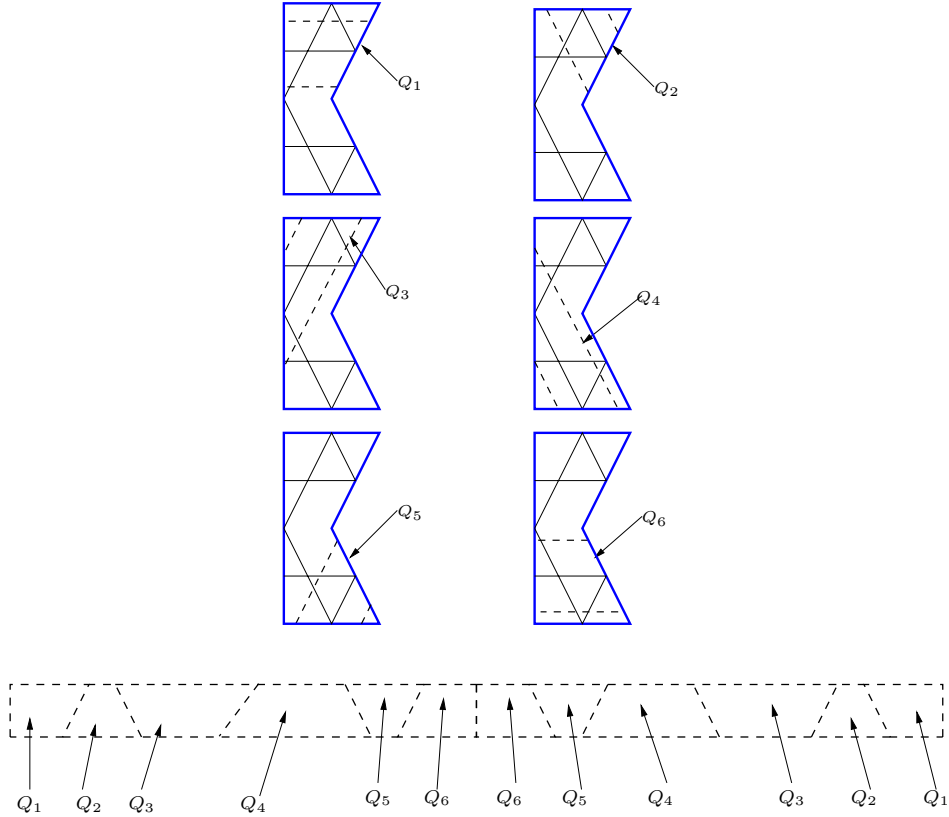}
\caption{Construction of a cylinder from its components $Q_j$ along a periodic trajectory.}
\label{fig:2}
\end{figure}

Let $\eta \in (-\delta, \delta)$. As $\eta \uparrow \delta$, the periodic geodesics $h(t,\eta)$ converge 
to a possibly singular periodic geodesic (and similarly for $\eta \downarrow -\delta$). 
The cylinder $\cyl_{\delta,T}$ will be called maximal if both these geodesics are singular. 
In the geometric condition $(CC)$, we may assume that the cylinders are maximal.
  
We now define the Euclidean Laplace operator on a \ESCS.  First note that the Euclidean metric 
on $X$ provides us with a well-defined $L^2$ norm and that smooth functions compactly 
supported in $X_0$ are dense in $L^2(X)$. For any such function, we can also 
define the quadratic form $q(u)= \int_{X} |\nabla u|^2 dx$ in which $\nabla$ is taken with respect 
to the Euclidean metric and $dx$ is the Euclidean area element. The Laplace operator is 
the self-adjoint operator associated with the closure of this quadratic form. It is also the 
Friedrichs extension of the usual Euclidean Laplace operator defined on ${\mathcal C}^\infty_0(X_0).$ 
It is standard that this operator has compact resolvent so that its spectrum is purely discrete
and we may consider its eigenvalues and eigenfunctions.
 
%%%%%%%%%%%%%% Semiclassical measures on X_0 %%%%%%%%%%%%%%

Let $u_n$ be a sequence of eigenfunctions on $X$ associated with a sequence of eigenvalues going to infinity. 
We want to associate to this sequence a so-called {\em semiclassical measure}. Since we do not want to 
look precisely at what is happening at the conical point, our semiclassical measure $\mu$ will 
be a positive distribution acting on ${\mathcal C}_0^\infty(S^*X_0),$ 
where $S^*X_0$ denotes the unit cotangent bundle over $X_0.$ 
Our semiclassical measure is then given by the usual recipe. In particular, for any $a\in {\mathcal C}_0^\infty(S^*X_0)$ 
and any zeroth-order pseudodifferential operator $A$ on $X$ with principal symbol $a$ we have 
$$
\lim_{n \to \infty} \langle Au_n,u_n\rangle \,=\, \int_{S^*X_0} ad\mu.
$$

\begin{rem}
It is considerably simpler to define a pseudodifferential operator 
on $X_0$ than on $X$. In particular we may use local isometries with the Euclidean plane. 
\end{rem}
  
\begin{rem}\label{ProbMeas}
It should be noted that, in contrast with the usual semiclassical measure, with this definition, 
a semiclassical measure need not be a probability measure. In order to be a probability 
measure one has to prove that, loosely speaking, no mass accumulates at the conical points. 
\end{rem}

The invariance property of this measure by the geodesic flow also has to be taken carefully. 
The infinitesimal version of this invariance is true using the standard commutator argument and 
Egorov's theorem since this computation takes place integrally over $X_0.$ One can then integrate 
this property along any 
geodesic until it reaches a vertex.

%%%%%%%%%%%%%%%%%%%%%%%%%%%%%%%%%%%

\section{Proof of the main theorem}
\label{three}

Let $U$ be a domain of $X$ satisfying condition $(CC)$.  Necessarily we have $P\subset U.$ 
We will denote by $U_0=U \backslash P$.  
Let $u_n$ be a sequence of normalized eigenfunctions such that $\int_U |u_n|^2\rightarrow 0.$ 
Let $\mu$ be any semiclassical measure associated to $(u_n)$. Then 
we have

\begin{lem}\label{supp}
(i) The support of $\mu$ is disjoint from $\pi^{-1}(U_0)$. 

(ii) $\mu$ is a probability measure that is invariant under the geodesic flow. 
\end{lem}

\begin{proof}
(i) Suppose that there is a point $q \in \supp \mu$ with $\pi(q) \in U_0$. Choose a nonnegative function $\phi \in C^\infty(X_0)$ supported in $U_0$, with $\phi \equiv 1$ in a small neighbourhood $G$ of $\pi(q)$. Since $\phi \geq 0$ and $\mu$ is a positive measure, we have $\langle \mu, \phi \rangle \geq 0$. If $\langle \mu, \phi \rangle = 0$ then $\langle \mu, \chi \rangle = 0$ for every $\chi \in C_0^\infty(S^*X_0)$ supported in $\pi^{-1}(G)$, since we have $\chi = \chi \phi$, and by the positivity of $\mu$ and $\phi$, $|\langle \mu, \chi \phi \rangle|$ is bounded by $\langle \mu, \phi \rangle \| \chi \|_\infty$. But this would mean that $\pi^{-1}(G)$ is disjoint from the support of $\mu$, which is not the case. Thus we conclude that $\langle \mu, \phi \rangle > 0$. This means that 
$$
\lim_{n \to \infty} \int_B |u_n|^2 \phi  > 0,
$$
contradicting our assumption about the sequence $(u_n)$. This proves (i). 

(ii) Statement (i) tells us precisely that no mass accumulates at the conical points and thus $\mu$ is 
a probability measure (see Remark \ref{ProbMeas} above). The invariance holds since $\mu$ is 
a semiclassical measure. 
\end{proof}

Let $\mu$ be as above, and let $(z, \zeta) \in T^* X_0$ be in the support of $\mu.$ 
According to the preceding lemma  and the invariance property of $\mu$, 
condition $(CC)$ implies that $z$ belongs to a cylinder periodic in the direction $\zeta.$ 

The support of $\mu$ is thus included in the union of the maximal cylinders 
$\cyl_i$ defined in condition $(CC).$

Let $\cyl$ be such a cylinder. By definition, there is a local isometry between 
$\SP^1_L\times (0,a)$ and $\cyl.$ Using it, we can pull-back the eigenfunction $u_n$ to 
$\cyl$.  
We now apply the argument of \cite{M1} to this function $u_n$ on $\cyl$. 
Let us use Cartesian coordinates $(x, y)$ on $\cyl$, where $x \in [0, L]$, $y \in [0, a]$ 
with $x = 0$ and $x=L$ identified. 
Thus $\{ y = 0 \}$ and $\{  y = a \}$ are the two long sides of the cylinder, 
and the variable $y$ parametrizes  periodic 
geodesics. Choose a cutoff function $\chi \in C_c^\infty[0, a]$ such that $\chi = 1$ on an open set containing 
all $y$ parametrizing all paths disjoint from $U_\epsilon$ 
(as opposed to $U_{\epsilon/2}$). Then $\chi u_n$ vanishes near the long 
sides of $\cyl$, and thus may be regarded as a function on a torus $T$. 
So we now have a sequence $v_n = \chi u_n$ on $T$. 
Consider any semiclassical measure $\nu$ associated with the sequence $(v_n)$ on $T$. 
By compactness, the $v_n$ are bounded in $L^2$, 
so there exists at least one semiclassical measure associated with $(v_n)$ (on the torus). 
This could be the zero measure; this would be the case if $\| v_n \|_{L^2} \to 0$, for example. 
Since $\mu$ is supported on a finite number of cylinders, there are only a finite number of directions in the 
support of $\nu$. So we can find a constant-coefficient pseudodifferential operator $\Phi$ on $T$ that is 
microlocally $1$ in a neighbourhood of directions 
parallel to $dx$, i.e. in the direction of the unwrapped periodic paths, 
but vanishes microlocally in a neighbourhood of 
 every other direction in the support of $\nu$. 
 (See \cite{M1} for a discussion of constant-coefficient pseudodifferential operators on a torus.) 
 
Consider the sequence of functions $(\Phi v_n)$ on $T$. The semiclassical measures 
$\nu'$ associated to this sequence are related 
to those for the sequence $(v_n)$ by $\nu' = \sigma(\Phi) \nu$. 
Thus, the support of $\nu'$ is restricted to directions parallel 
to $dx $ and to geodesics parametrized by $y$ such that  $\chi(y) = 1$ 
(because of lemma \ref{supp} and the way we chose $\chi$). 
 
 Now we apply the proposition on p46 of \cite{BZ3} which says:

\begin{prop} Let $\Delta = -(\partial_x^2 + \partial_y^2)$ be the Laplacian 
on a rectangle $R = [0,l]_x \times [0,a]_y$. 
 For any open $\omega \subset R^2$ of the form $[0,l]_x \times \omega_y$, 
there is $C$ such that, for any solution of 
 $$
 (\Delta - \lambda^2) w = f + \partial_x g 
 $$
 on $R$, satisfying periodic boundary conditions, we have
 $$
 \| w \|_{L^2(R)}^2 \leq C \Big( \| f \|_{L^2(R)}^2 +  \| g \|_{L^2(R)}^2 +  \| w \|_{L^2(\omega)}^2 \Big).
 $$\end{prop}
(This proposition is stated in \cite{BZ3} for Dirichlet boundary conditions on a rectangle, 
but applies equally well to periodic boundary conditions as noted in \cite{M1}.)
 We apply this with $w = w_n = \Phi v_n$, $f = f_n = \Phi((\partial_y^2 \chi) u_n)$, 
$g = g_n =  -2\Phi( (\partial_y \chi) u_n)$, and $\omega$ contained in the set 
$\{ \chi = 0 \}$. (Note that $\Phi$ commutes with $\Delta$ and $\partial_y$.)
 Since $f$ and $g$ are supported on the support of $\nabla \chi$, 
their support is disjoint from that of $\nu'$, so 
 $\| f_n \|_{L^2(R)}^2 +  \| g_n \|_{L^2(R)}^2 \to 0$ as $n \to \infty$. Also, by our choice of $\omega$, we have 
 $\| w_n \|_{L^2(\omega)}^2 = 0$. It follows that $\| w_n \|_{L^2(R)}^2 \to 0$. But this means that $\nu' = 0$. 
 This implies that $\nu$ has no mass along directions parallel to $dx$, which means that $\mu$ has no mass along 
 the cylinder $\cyl$. Since $\cyl$ is arbitrary, 
 and the number of such cylinders is finite, this means that $\mu$ has no mass, i.e. it is the zero measure.  
 This contradicts part (ii) of Lemma~\ref{supp}.
 We conclude that Theorem~\ref{ThmMain} holds.
 
 %%%%%%%%%%%%%%%%%%%%%%%%%%%%%%%%%%%

\section{Condition $(CC)$ for neighbourhoods of the conic set $P$}
\label{four}

In this section we study condition $(CC)$ in more detail. We wish in particular to address 
this condition for the region $U_\epsilon$ which is the $\eps$ neighbourhood of the set $P$ of conical points of an ESCS $X$. 

In this case we have the following proposition.

\begin{prop}\label{GKTDichotomy}
Let $X$ be an orientable \ESCS \ with singular set $P$. 

(i) For any geodesic $\gamma$, either $\gamma$ is periodic, or the closure of  $\gamma$ meets $P$. 

(ii) Let $U_\epsilon$ denote the $\eps$ neighbourhood of $P$. 
Then 
any periodic geodesic avoiding $U_\eps$ (which is periodic by part (i)) belongs to a maximal cylinder and the number of such maximal cylinders is finite.
\end{prop}

Before proving Proposition~\ref{GKTDichotomy}, we introduce some notation and definitions. 
A {\em strip} is an isometric immersion
$h~ :~ \R \times I\rightarrow X_0,$ where $I$ is a nonempty open interval and 
$\R\times I$ is equipped with the Euclidean metric. 
We will also sometimes call the image of $h$ a strip. 
The {\em width} of the strip is 
the length of the interval $I.$ 

For any strip, the mappings $\gamma_c:=h(\cdot,c)$, $c \in I$, are geodesics 
of $X_0$. Since $h$ is a local isometry and $X$ is orientable, 
if one $\gamma_c$ is periodic of length $L$ then, for any $c',$ 
$\gamma_{c'}$ is also periodic with length $L$.  

A {\em maximal} strip is a strip that cannot be extended 
to $\R\times I'$ for any open $I'$ properly containing $I$.  A strip is maximal if and only if 
$P$ intersects the closure of $h(\R\times I)$ on its left and on its right. 

For any geodesic $\gamma$ we will denote by $\tilde{\gamma}$ the geodesic lifted 
to the unit tangent bundle $SX_0$ and we denote by $\pi$ the projection 
of $SX_0$ in $X.$ We also denote by $d(.,.)$ the distance on $X.$

Proposition \ref{GKTDichotomy} is a straightforward consequence of the following lemma, which is closely related to results of  \cite{GKT}.

\begin{lem}\label{GKT}
Let $h:\R\times I\,\rightarrow \, X_0$ be a strip of positive width. Then there 
exists $L$ such that $h(t+L,s)=h(t,s).$
\end{lem}

\begin{proof}
We follow closely the ideas of \cite{GKT}. 
We may assume that $h$ is maximal and $I=(-\delta,\delta)$ 
and we will prove that $\gamma_0$ is periodic. Observe that 
for any $t$, $d(\gamma_0(t),P)\geq \delta.$ We denote by $Z\subset SX_0$ the forward 
limit set of the lifted geodesic $\tilde{\gamma}_0.$ By continuity,we have that 
$d(\pi(Z),P)\geq \delta.$ This implies first that $Z$ is compact and then 
that the geodesic flow is continuous on $Z.$ Using Furstenberg's uniform 
recurrence theorem \cite{Furst}, there exists a point $x$ 
that is uniformly recurrent in $Z.$ We denote by $G$ the geodesic emanating from $x$
(observe that ${\tilde{G}}(\R)\subset Z$). 
We also denote by $H: \br \times (-\Delta^-,\Delta^+)$ 
the maximal strip around $G.$ 
Uniform recurrence means the 
following: for any neighbourhood $\tilde{W}\subset SM_0$ that intersects 
$\tilde{G}(\R)$, there exists $L\in\R$ such that 
$$
\forall t,\exists s\in [t,t+L]~~\mbox{such that}~~(G(s),\dot{G}(s))\in \tilde{W}.
$$

The uniform recurrence and the maximality of 
the strip imply 
\begin{equation}\label{UniRec}
\forall\, \eps>0,\,\exists\, L~~\mbox{such that}~~ \forall \, t,~ 
d(H([t,t+L]\times\{\Delta^+-\eps\}),P)<2\eps.
\end{equation}
 
Indeed, by maximality, for any $\eps>0$, there exists $t_0$ such that the geodesic $\overline\gamma$ emanating from the point in phase space given by
$(G(t_0),\dot{G}(t_0)-\frac{\pi}{2})$
hits a conical point in time less than $\Delta^++\frac{\eps}{2}.$ 
In particular 
$$d(H(t_0,\Delta^+-\eps),P)\leq \frac{3\eps}{2}.$$
By continuity, we can find a neighbourhood $\tilde{V}$ of 
$(G(t_0),\dot{G}(t_0))$ such that, for any $(m,\theta)$ 
in this neighbourhood, the geodesic starting from $(m,\theta-\frac{\pi}{2})$ 
stays in the $\frac{\eps}{2}$ tubular neighbourhood of $\overline\gamma$ until time 
$\Delta^+-\eps.$ Using uniform recurrence, there exists $L$ such that, 
for any $t$, there exists $s\in [t,t+L]$ so that $(G(s),\dot{G}(s))$ belongs to 
$\tilde{V}.$ Using the preceding property we have that 
$$d(H(s,\Delta^+-\eps),H(t_0,\Delta^+-\eps))<\frac{\eps}{2}.$$  We conclude (\ref{UniRec}) using 
the triangle inequality.

Observe that (\ref{UniRec}) means that 
if we represent the strip $H$ by a vertical strip in 
$\R^2$ then we can find a vertical 
strip of width $2\eps$ (that contains the right boundary of the 
strip $H$) such that any rectangle of height $L$ contained in 
this strip contains at least one conical point.
 
We fix local coordinates near $x$ so that $x=((0,0),\frac{\pi}{2}).$ 
Since $x$ is in the forward limit set, 
there exists $t_n$ such that $\tilde{\gamma}_0(t_n)$ converges to $x.$ 
We set $\tilde{\gamma}_0(t_n)=(z_n,\theta_n).$ Represent now in 
$\R^2,$ the strip $H$ around $(0,0)$ and the strip 
$h$ around $z_n.$ By standard Euclidean geometry, for $\epsilon  < < \delta$ 
the intersection of any vertical strip of width $2\eps$ with $h$ contains a 
vertical rectangle of width $2\eps$ and height that goes to $\infty$ 
when $\theta$ goes to $\frac{\pi}{2}$ (see Figure~\ref{strips}).  Indeed, denoting by $\alpha=\frac{\pi}{2}-\theta$, this 
height is $$\frac{2\delta}{\sin|\alpha|}-\frac{2\eps}{\tan|\alpha|}.$$ 
Using (\ref{UniRec}), since there is no conical point in $h,$ this implies that 
we have $\theta_n=\frac{\pi}{2}$ for $n$ large enough. The strip $h$ 
around $z_n$ is thus represented by a vertical strip of width $2\delta.$ %If two such strips do not exactly overlap then the boundary of one 
%strip is in the interior of the second one which is impossible by maximality. Thus for $n$ large enough, $z_n$ and $z_{n+1}$ have the same abscissa and this implies that the geodesic $\gamma_0$ is periodic.  
Then maximality of both $h$ and $H$ implies that the strips coincide up to a translation in the first variable; in particular, the widths coincide, and $z_n$ is independent of $n$. But this implies that the geodesic $\gamma_0$ is periodic.  
\end{proof}

\begin{figure}
\includegraphics[height=10cm]{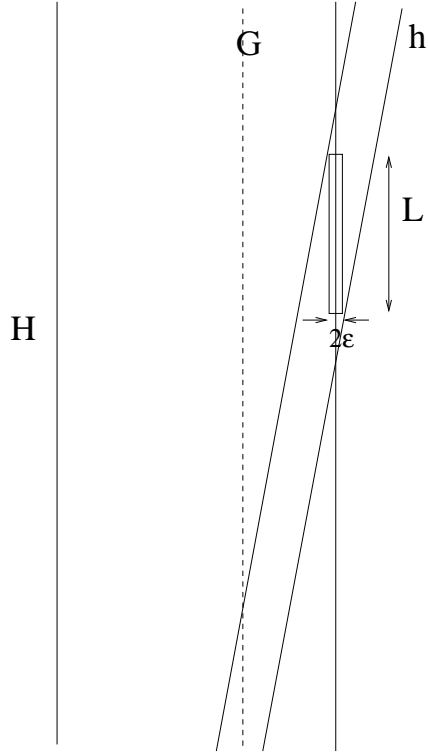}
\caption{Illustration of the argument in the proof of Lemma~\ref{GKT}. For sufficiently large $n$, if $\theta_n \neq \pi/2$ then the strip $h$ would contain a rectangle of size $2\epsilon \times L$, as illustrated. This is not possible as any such rectangle  intersects $P$.}
\label{strips}
\end{figure}

\begin{proof}[Proof of Proposition~\ref{GKTDichotomy}]

(i) Let us consider a geodesic $g$ 
such that $\bar{g(\R)}$ contains no conical point. There 
exists $\epsilon>0$ such that $\forall t, B(g(t),\epsilon)\cap P =\emptyset.$
For, otherwise, we could find sequences $\epsilon_n,~t_n,~$ and $p_n$ such that 
$d(g(t_n),p_n)< \epsilon_n$ contradicting the hypothesis. 
This implies that the geodesic $g~:~\R \rightarrow X_0$ extends to a strip 
of positive width and Lemma~\ref{GKT} concludes the proof. 

(ii) We only have to prove the finiteness property. Denote by $\cyl_i$ the maximal 
cylinders. By definition the {\em middle geodesic} of $\cyl_i$ is at distance at 
least $\eps$ of the conical points. 
So that the $\eps/2$ strip around this geodesic consists in periodic geodesics at distance at 
least $\eps/2$ of the conical 
points. Denote by $S_i$ this strip. 
The proof of Lemma 4.2 of \cite{Hil} (see also fig. $2$ of this reference 
\cite{DG}) implies that if $\gamma_i$ and $\gamma_j$ 
are two periodic geodesics in strip $S_i$ and $S_j$ respectively then, at any of 
their intersections, they make an angle $\theta$ satisfying 
$$\frac{1}{\sin \theta} \leq  \frac{\min(L_i,L_j)}{\eps}.$$   
We can now adapt the argument of \cite{DG}. 
Fo any $i$ we consider the following region $V_i$ of $SX_0,$ 
$$V_i =\{ (x,\theta),~x\in S_i, |\theta-\theta_i|< \frac{\eps}{2L_i}\}.$$
The preceding estimate implies : 
\begin{enumerate}
\item $V_i$ is isometric to $(-\eps/2,\eps/2)\times \SP_{L_i}\times (-\frac{\eps}{2L_i},\frac{\eps}{2L_i}),$ 
\item for different cylinders, the regions $V_i$ are distinct.
\end{enumerate}
The first point implies that the volume of $V_i$ is bounded away from zero independently of the cylinder, and 
the second point coupled with the fact that the unit tangent space to $X$ has finite volume yields the result.
\end{proof}

\begin{rem}
We have seen that the $\eps$ neighbourhood $U_\epsilon$ of the conical points is a control region for $X_0$. 
Since any geodesic that enters the $\eps$-neighbourhood also enters the 
annular region $\eps/2 \leq d(x,P) \leq \eps$, the union of these annular regions is also a control region. 
A similar argument also shows 
that for any $\delta>0$ and any $\eps>0$ the $\eps$-neighbourhood of the union of the normal co-bundle to 
the circles $d(x,p)=\delta$ is a control region (in the cotangent bundle!).
\end{rem}

\end{document}